\documentclass[11pt]{article}
\usepackage{amssymb, amsmath, amsthm}
\usepackage{epsfig}
\numberwithin{equation}{section}
\newtheorem{teo}{Theorem}[section]
\newtheorem{cor}[teo]{Corollary}
\newtheorem{prop}[teo]{Proposition}
\newtheorem{lemma}[teo]{Lemma}

\newtheorem{remark}[teo]{Remark}
\newtheorem{example}[teo]{Example}
\newtheorem{definition}[teo]{Definition}

\theoremstyle{definition}

\theoremstyle{remark}

 \font\filt=msbm10
\def\neweq#1{\begin{equation}\label{#1}}
\def\endeq{\end{equation}}

\def\proof{\noindent{\it Proof.} }
\def\ds{\displaystyle}

\def \R {\hbox{\filt R}}

\def\H{\mathcal{H}}
\def\Cap{{\rm Cap}}
\def\E{\mathcal{E}}
\def\K{\mathcal{K}}
\def\M{\mathcal{M}}
\def\S{{\Sigma}}
\def\C{\mathcal{C}}

\def\J{\mathcal{J}}

\def\ds{\displaystyle}

\def \res{\mathop{\hbox{\vrule height 7pt width .5pt depth 0pt
\vrule height .5pt width 6pt depth 0pt}}\nolimits}

\def\B{\mathcal{B}}

\def \res{\mathop{\hbox{\vrule height 7pt width .5pt depth 0pt
\vrule height .5pt width 6pt depth 0pt}}\nolimits}

\begin{document}

\title{Optimal convex shapes for concave functionals}

\author{Dorin BUCUR, Ilaria FRAGAL\`A, Jimmy LAMBOLEY}

\date{\today}

\baselineskip14pt \maketitle

\begin{abstract}
Motivated by a long-standing conjecture of P\'olya and Szeg\"{o} about the Newtonian capacity of convex bodies, we discuss the role of concavity inequalities in shape optimization, and we provide several counterexamples to the Blaschke-concavity of variational functionals, including capacity.
We then introduce a new algebraic structure on convex bodies, which allows to obtain global concavity and  indecomposability results, and we discuss their application to isoperimetric-like inequalities. As a byproduct of this approach we also obtain a quantitative version of the Kneser-S\"uss inequality.
Finally, for a large class of functionals involving Dirichlet energies and the surface measure,  we perform a local analysis of strictly convex portions of the boundary via second order shape derivatives. This allows in particular to  exclude the presence of smooth regions with positive Gauss curvature in an optimal shape for P\'olya-Szeg\"o problem.
\end{abstract}

\medskip
\noindent {\small {\it 2000MSC\ }: 49Q10, 31A15.}


\section{Introduction}
The initial motivation of this paper
may be traced in a long-standing conjecture by P\'olya and Szeg\"{o}  on minimal capacity sets. Precisely, the question is to find the convex sets in $\R^3$ of prescribed surface area which minimize the electrostatic capacity. In \cite{posz51} it is conjectured that the minimizer is a planar disk, and the problem is currently open, though some advances in favor of the conjecture have been recently made in \cite{CFG, FrGaPi}.
From the point of view of shape optimization, the most interesting feature is that the expected optimal shape is degenerate (the optimality of the disk among {\it planar} domains being straightforward).
Hence one may reasonably expect that capacity enjoys a certain concavity-like property with respect to the shape variation.

\medskip
Our first purpose is actually to investigate global concavity inequalities for functionals defined on
the family of convex sets, in the perspective of applying them to solve shape optimization problems.
The underlying idea is that
minimizing a concave shape functional leads in a natural way to optimal sets
which are either degenerate or, in a suitable sense, extremal.
Of course, the notions of concavity and extremality become meaningful only after
the family of convex sets has been endowed with a certain algebraic structure.
The choice of such algebraic structure is a crucial step, which requires a careful
balance between two opposite purposes: make the shape functional concave and the class of extremal convex bodies as restricted as possible.

\medskip
The algebraic structures most commonly considered on convex bodies are the Minkowski and Blaschke additions, which agree for planar sets but are deeply different in dimension $n \geq 3$. We refer the reader
to Section \ref{2} for their definitions, and to \cite{ku02,ku09,Alex} for their systematic use in solving isoperimetric-type problems (in a similar spirit, see also the recent papers \cite{cagr07}, where global convexity properties associated to shadow systems allow to understand the extremality of parallelograms for the Mahler conjecture, and \cite{lano10}, where polyhedra appear as optimal shapes).
The best choice among the Minkowski or Blaschke sum is self-evident in many situations, for instance in presence of a constraint which behaves linearly with respect to one of them: respectively,  a prescribed mean width or surface area.

\medskip
In the Minkowski structure,  the queen mother among concavity inequalities is the Brunn-Minkowski theorem for the volume,  a cornerstone in Convex Geometry dating back to one century ago, for which we refer to the extended survey paper \cite{Ga}.
 A lot of efforts have been made during the years to show that the analogous concavity inequality holds true for several variational energies of Dirichlet type. A non-exhaustive list includes (functionals of) the first Dirichlet-eigenvalue of the Laplacian \cite{BrLi}, the Newtonian capacity \cite{Bo,CaJeLi},
 the torsional rigidity \cite{Bo2}, the $p$-capacity \cite{CoSa},
 the logarithmic capacity \cite{CoCu},
the first eigenvalue of the Monge-Amp\`ere operator \cite{Sa}, the Bernoulli constant \cite{bisa09}.
For all these functionals, the Brunn-Minkowski inequality, combined with a result of Hadwiger, allows to identify balls as maximizers
under prescribed mean width. Under the same constraint,
minimizers must be searched among ``extremal'' sets in the Minkowski
structure, usually called indecomposable bodies. Except in dimension $n=2$, when such bodies are reduced to (possibly degenerate) triangles, this is a very weak information since  Minkowski-indecomposable sets are a dense family for $n \geq 3$.

\medskip
On the contrary, in the Blaschke structure,
a theorem of Bronshtein
characterizes simplexes as the unique indecomposable bodies (in any dimension).
For this reason, it is interesting to understand which kind of functionals enjoy a Blaschke-concavity property: their  only possible minimizers under prescribed surface area
will be either degenerate sets or simplexes. To the best of our knowledge, the unique known result in this direction is a theorem proved in the thirties by Kneser and S\"uss. It states a concavity inequality for the volume functional, which reads as the exact analogue of Brunn-Minkowski theorem in the Blaschke structure.

\medskip
The starting point  of this paper is actually the investigation of Kneser-S\"uss type inequalities for other functionals than volume, such as capacity and the first Dirichlet eigenvalue of the Laplacian. Quite surprisingly, we arrive fast at the conclusion that  Blaschke-concavity inequalities for those functionals do not hold (see Propositions \ref{rem1} and \ref{rem3}).
In particular we learn that,
regarding concavity, capacity behaves differently from volume,
though
in the recent paper \cite{GaHa} the authors claim that it
has the same ``status'' as volume, as it plays  the role of its own dual set function in the dual
Brunn-Minkowski theory.  The argument we use is very simple, and relies on a Hadwiger-type theorem which allows to identify balls as maximizers of
any Blaschke-concave functional under a surface area constraint (see Theorem \ref{KSteo}).
Incidentally, by the same method we provide a negative answer to the open question, stated for instance in \cite{Co},
 whether
the second Dirichlet eigenvalue of the Laplacian satisfies a Brunn-Minkowksi type inequality (see Proposition \ref{rem2}).

\medskip
Moving from these counterexamples, in order to gain concavity for functionals
other than volume,  we introduce a new algebraic structure on convex bodies.
The reason why we drop the Blaschke addition is that it amounts to sum the surface area measures of two convex bodies, and the surface area measure is precisely related to the first variation of volume. In order to deal with a given functional different from volume, the natural idea is to sum its first variation measure at two given sets. Of course, some assumptions are needed to ensure the well-posedness of the notion (in the same way as Minkowski existence theorem must be invoked in the definition of Blaschke sum, see Definitions \ref{crucial1} and \ref{crucial2} for more details).
However this can be done in good generality, and the abstract framework includes several significant examples, including capacity, torsional rigidity, and first Dirichlet eigenvalue of the Laplacian. With respect to this new sum, we prove that concavity holds true for
any functional which satisfies the Brunn-Minkowski inequality and has first variation representable by a suitable integral formula: we believe that such result (Theorem \ref{abstract1}) can be regarded as the natural analogue of Kneser-S\"uss theorem for the volume functional. We also show that our approach allows to deduce  from the results in \cite{fmp09}   a quantitative version of the Kneser-S\"uss inequality which involves the Fraenkel asymmetry (see Corollary \ref{ksqq}).

Next, still in the same algebraic structure, we are able to characterize indecomposable bodies in any dimension (see Theorem \ref{B}).
We then turn attention to the resulting isoperimetric-type inequalities in shape optimization: some of them (like (\ref{known})) are
rediscovered from the literature, some others (like (\ref{unknown})) are new, and some others (like (\ref{open})) rise delicate open questions.
Unfortunately, the above described results do not help to solve the P\'olya and Szeg\"{o} conjecture, since the surface constraint is difficult to handle in the new algebraic structure.

\medskip
Thus, with a change of perspective, we turn attention to investigate {\it local} properties of the optimal convex shapes by the analysis of the second order shape derivative. Contrary to the above discussed global concavity properties,
in the local analysis of a strictly convex region we are allowed to use  deformations which are not necessarily a resultant of the sum of convex sets. As usually, second order shape derivatives are highly technical and rather difficult to manage. Nevertheless, when computed at an optimal shape, they give an interesting information. Precisely, for a large class of functionals involving Dirichlet energies and the surface measure, and in particular for the P\'olya-Szeg\"{o} capacity problem, we prove that an optimal shape cannot have smooth open regions on its boundary with a non vanishing Gauss curvature (see Remark \ref{ex:PS}). This can be done by observing that, otherwise,  the second order shape derivative would not be a positive bilinear form with respect to a suitable class of shape deformations. We point out that the nonexistence of smooth regions with strictly positive curvatures, though of course does not prove entirely P\'olya-Szeg\"o conjecture, is the first qualitative information available on optimal sets: by now, the study of the problem had been carried out essentially by means of the first variation, which apparently does not contain enough information.

\medskip
To conclude this introduction, we wish to mention two open questions which seem of relevant interest
among those left by this paper: finding some methodology to exclude the  optimality of polyhedra for the P\'olya-Szeg\"o problem, and
understanding whether a Blaschke-concavity inequality for capacity may hold true within some restricted class of convex bodies.

\medskip
The contents are organized as follows.

In Section \ref{2} we discuss concavity inequalities in the Minkowski and the Blaschke structure.

In Section \ref{3} we define the new algebraic structure, and we give the related concavity and indecomposability results.

In Section \ref{4} we deal with the local approach via second order shape derivatives.

\section {Global Minkowski and Blaschke-concavity inequalities}\label{2}
We start by introducing, in Sections \ref{width} and \ref{surface} below, the Brunn-Minkowski and Kneser-S\"uss inequalities.
This is done for convenience of the reader with a twofold aim:
showing how these inequalities can be used in shape optimization problems, and paving the way
to the counterexamples contained in Section \ref{counter}.

Throughout the paper, we adopt the following notation: we let  $\K ^n$ be the class of convex bodies (compact convex sets) in $\R ^n$,
$\K _0 ^n$ be the class of convex bodies with nonempty interior, $\B^n$ be the class of balls, and $\S^n$ be
the class of simplexes.

By saying that a functional $F:\K ^n \to \R ^+$ is $\alpha$-homogenous (for some $\alpha \neq 0$), we mean that  $F(tK) = t ^ \alpha F (K)$ for every  $K \in \K ^n$ and every $t \in \R^+$.

For $K \in \K ^n$, we recall that its support function $h (K)$ is defined on the unit sphere $\mathbb{S} ^ {n-1}$ of $\R^n$ by:
 $$h(K) (\nu):= \sup _{x \in K} \big ( x \cdot \nu \big )   \qquad \forall \nu \in \mathbb{S} ^ {n-1}\ .$$

Moreover, for $K \in \K _0 ^n$, we denote by $\nu _K: \partial K \to \mathbb{S} ^ {n-1}$ the Gauss map of $K$ (which is well defined $\H ^ {n-1}$-a.e.\ on $\partial K$).

\subsection{Extremal problems under mean width constraint}\label{width}

The Minkowski addition of two convex bodies $K$ and $L$ can be formally defined as the convex body $K+L$ such that
$$h(K+L)=h(K)+h(L)\ .$$

If a functional $ F: \K ^n \to \R ^+$ is Minkowski linear (meaning that $F(K+L)= F(K) + F(L)$),
continuous in the Hausdorff distance and rigid motion invariant, then it is a constant multiple of the mean width $M(K)$ (see \cite[p.167]{Sch}).
Recall that $M (K)$ is defined by
$$M (K) := \frac{2}{\H ^ {n-1} (\mathbb{S}^ {n-1})} \int _{\mathbb{S} ^ {n-1}} h (K) d \H ^ {n-1}\ , $$
and coincides with the perimeter of $K$ if we are in dimension $n=2$.

\medskip
We say that a $\alpha$-homogenous functional $F: \K ^n \to \R ^+$ satisfies the Brunn-Minkowski inequality if
\begin{equation}\label{BM}
F^{1/\alpha} (K +L) \geq F ^ {1/\alpha} (K ) + F ^ {1/\alpha} (L) \qquad \forall \, K, L \in \K ^n\ .
\end{equation}

Besides volume (see \cite[Theorem 6.1.1]{Sch}), several functionals satisfy the Brunn-Minkowski inequality, see the list of references given in the Introduction and also the survey paper  \cite{Co}.
To any of these functionals we may apply the following result. We give the proof for completeness since we were unable to find a precise reference
in the literature (except for the sketch given in \cite[Remark 6.1]{bisa09}); for similar results, see also the works \cite{Ha} and \cite{Sc}.

\medskip

\begin{teo}\label{BMteo}
Assume that $F: \K ^n \to \R ^+$ is $\alpha$-homogeneous, continuous in the Hausdorff distance,
invariant under rigid motions, and satisfies the Brunn-Minkowski inequality.
Consider the functional
$$\E(K) := \frac{F^ {1 / \alpha}(K)}{M(K)}\ .$$
Then:
\begin{itemize}
\item[--] the maximum of $\E$ over $\K ^n$ is attained on $\B^n$;
\item[--] if inequality \eqref{BM} is strict for non-homothetic sets, then 
$\E$ attains its maximum over $\K ^n$ only on $\B^n$; moreover, 
for $n=2$, $\E$ can attain its minimum over $\K ^2$ only on $\S^2$ (triangles) or on $\K ^2 \setminus \K ^2 _0$ (segments).
\end{itemize}
\end{teo}

\proof  By a theorem of Hadwiger
\cite[Theorem 3.3.2]{Sch}, for any $K \in \K ^n$ with  affine hull of strictly positive dimension, there exists a sequence of Minkowski rotation means of $K$ which converges to a ball in the Hausdorff metric. Since $F$ satisfies (\ref{BM}) and the mean width is Minkowski linear,
using also the continuity and invariance assumptions on $F$, it follows immediately that balls are maximizers for the quotient functional $\E$, 
and that no other maximizer exists if the strict inequality holds in (\ref{BM}) for non-homothetic sets. 
On the other hand, if such a strict inequality holds, 
the minimum of the shape functional $\E$ is attained necessarily on a Minkowski-indecomposable set. Indeed, if $K$ can be decomposed as
$K = K' + K''$, with $K'$ and $K''$ non-homothetic, then $K$ cannot be a minimizer for $\E$, since
$$
\begin{array}{ll}\displaystyle{\E (K) = \frac{F^ {1/\alpha} ( K' +  K'' )} {M ( K' +  K'' )}}
   & >
\displaystyle{   \frac{F^ {1/\alpha} ( K' ) + F^ {1/\alpha} (  K'' )} {M ( K')  + M (K'' )} }
   \\  & \displaystyle{\geq \min _{i=1,2}
\Big \{ \frac{F^ {1/\alpha} ( K' ) } {M ( K') }
  \ ,\
\frac{F^ {1/\alpha} ( K'' ) } {M ( K'') }
    \Big \}}\ .
\end{array}
$$
Then the shape of possible minimizers for $\E$ over $\K ^n$ can be deduced from the identification, holding in dimension
$n=2$, of Minkowski indecomposable sets with (possibly degenerate)  triangles \cite[Theorem 3.2.11]{Sch}.  \qed

\begin{remark} {\rm (i) If $n\ge 3$, the family of Minkowski indecomposable sets
is dense in $\K ^n$  \cite[Theorem 3.2.14]{Sch}, so that arguing as above no qualitative information about minimizers of $\E$ can be extracted.

\smallskip
(ii)  Clearly the shape functional $\E$ is invariant by dilations, and the problem of maximizing or minimizing $\E$ over $\K ^n$ is equivalent to the problem of maximizing or minimizing $F$ under a prescribed mean width. A similar remark can be repeated for the shape optimization problems considered in the next sections.
}
\end{remark}

Thanks to Theorem \ref{BMteo},
we immediately get for instance that the maximum of the Newtonian capacity in $\K ^3$, under the constraint of prescribed mean width, is attained at the ball.

\subsection {Extremal problems under surface constraint}\label{surface}

Since in shape optimization problems, a surface area constraint occurs more frequently than a mean width one, it may be convenient to work with the Blaschke addition in place of the Minkowski one.

\medskip
We recall that the surface area measure $\mu (K)$ of a convex body $K \in \K _0 ^n$ is
the positive measure on  $\mathbb{S} ^{n-1}$ defined by
\begin{equation}\label{surfmeasure}
\mu (K)[\omega] = \H ^ {n-1} \big (  \nu _K ^ {-1}(\omega) \big ) \qquad
\hbox{ for all Borel sets } \omega \subseteq \mathbb{S} ^ {n-1}\ .
\end{equation}
For instance, if $K$ is a polyhedron with faces $F_i$
and normals with endpoints $P_i$, then $\mu (K) = \sum _i \H ^ {n-1} (F_i) \delta _{P_i}$, whereas if $K$ is smooth and has strictly positive curvatures, then $\mu (K) = (G_K \circ \nu_K ^ {-1}) ^ {-1}\H ^ {n-1} \res \mathbb{S} ^ {n-1}$, being $G_K$ the Gaussian curvature of $K$.

\medskip
The Blaschke addition of two elements $K$ and $L$ of $\K _0 ^n$
is defined as the unique convex body  $K\stackrel{.}{+}L \in \K _0 ^n$ such that
$$\mu (K \stackrel{.}{+} L) = \mu(K)+\mu(L)\ . $$
(Accordingly, one may also define the Blaschke product $t \cdot K$ for any $t \in \R^+$ and $K \in \K _0 ^n$ through the identity
$\mu (t \cdot K) = t \mu (K)$, which  amounts to say that $t \cdot K = t ^ {1/ (n-1)} K$.)

About the correctness of this definition, we refer the reader to the comments in Section \ref{3}, Example \ref{ex}.

\medskip
If a functional $ F: \K_0 ^n \to \R ^+$ is Blaschke linear (meaning that $F(K\stackrel{.} +L)= F(K) + F(L)$),
continuous in the Hausdorff distance and rigid motion invariant, then it is a constant multiple of the surface area $S(K)$ (see \cite[Theorem 10.1]{GrZh}).

\medskip
By a slight abuse of language, and referring to the algebraic structure on the family of convex sets,
we say that a $\alpha$-homogenous functional $F: \K ^n \to \R ^+$ satisfies the Kneser-S\"uss inequality, if

\begin{equation}\label{falseconj}
F^{(n-1)/\alpha} (K \stackrel{.}{+} L) \geq F ^ {(n-1)/\alpha} (K ) + F ^ {(n-1)/\alpha} (L) \qquad \forall \, K, L \in \K _0 ^n\ .
\end{equation}

\bigskip
\begin{teo}\label{KSteo}
Assume that $F: \K ^n \to \R ^+$ is $\alpha$-homogeneous,  continuous in the Hausdorff distance, invariant under rigid motions,
and satisfies the Kneser-S\"uss inequality.
Consider the functional
$$\E(K) := \frac{F^ {(n-1) / \alpha}(K)}{S(K)}\ .$$
Then:
\begin{itemize}
\item[--] the maximum of $\E$ over $\K ^n$ is attained on $\B^n$;
\item[--] if inequality \eqref{falseconj} is strict for non-homothetic sets, then $\E$ attains its maximum over $\K ^n$ only on $\B ^n$; moreover, $\E$ can attain its minimum over $\K ^n$ only on $\S^n$ or on $\K ^n \setminus \K ^n _0$.
\end{itemize}
\end{teo}
\proof One can  follow the same line as in the proof of Theorem \ref{BMteo}. To obtain the maximality of balls, one has just to use, in place of Hadwiger theorem, an analogue result in the Blaschke structure \cite[Corollary 9.3]{GrZh}. To obtain information on possible minimizers, one has to apply the characterization of simplexes as the unique Blaschke-indecomposable bodies in $\K _0 ^n$ proved in \cite[Theorem 1]{Br}. \qed

\medskip

\begin{remark}{\rm In dimension $n=2$, the statements of Theorems \ref{BMteo} and \ref{KSteo} coincide, because the Blaschke and Minkowski addition, respectively the perimeter and the mean width, agree.} \end{remark}
\medskip

As a straightforward consequence of Theorem \ref{KSteo}, one gets for instance that:
\begin{itemize}
\item \label{ps2} under the constraint of prescribed perimeter, the maximum and the minimum of logarithmic capacity in $\K ^2$
are attained respectively at the
ball and at a possibly degenerate triangle;
\item  under the constraint of prescribed perimeter, the maximum and the minimum of  $p$-capacity in $\K ^2$, with $p \in (1,2)$,
are attained respectively at the ball and at a possibly degenerate triangle.
\end{itemize}

 Actually  for the logarithmic capacity it is known that the solution is a degenerate triangle, {\it i.e.} a segment \cite[p.51]{Po}.  Theorem \ref{KSteo}
 only says that the optimum is a possibly degenerate triangle. Further analysis is required to prove that the optimum is in fact  a segment.
 For the $p$-capacity, this question seems to be open.

 \medskip
Contrary to the Brunn-Minkowski inequality, as far as we know, the only functional which is known to satisfy the Kneser-S\"uss inequality is the volume functional (see \cite[Theorem 7.1.3]{Sch}). And in view of Theorem  \ref{KSteo} it is of relevant interest to understand which functionals do satisfy  the Kneser-S\"uss inequality.  In particular, if the Newtonian capacity in $\R^3$ would satisfy the Kneser-S\"uss inequality, then the solution to the  P\'olya-Szeg\"{o} conjecture would get a step forward. Indeed, in that case the optimal could be only either a tetrahedron or a two dimensional body. Since the disk is known to be optimal among planar domains (see \cite{CFG}), the analysis would be reduced to tetrahedrons, similarly as in the
above mentioned minimization problem for the logarithimic capacity in two dimensions.
 \bigskip

\subsection {Some non-concavity results}\label{counter}

We give below several negative results about the validity of the Kneser-S\"uss inequality $(\ref{falseconj})$ for variational functionals.

\begin{prop}\label{rem1}
In dimension $n=3$, the Newtonian capacity ${\rm Cap} (K)$ does not satisfy the  the Kneser-S\"uss inequality $(\ref{falseconj})$.
\end{prop}
\proof
If $(\ref{falseconj})$ would be true for $F(K)  = {\rm Cap} (K)$,
then by Theorem \ref{BMteo} the ball would have maximal capacity
among convex bodies with a prescribed surface area.
This contradicts the known fact that the quotient between the square of capacity and surface area diverges
along a sequence of thinning prolate ellipsoids (see \cite[Section 4]{CFG}).
\qed

\begin{prop}\label{rem2}
In dimension $n=2$, the second Dirichlet eigenvalue of the Laplacian, $\lambda _2 (K)$, does not satisfy
the Brunn-Minkowski inequality $(\ref{BM})$ (nor
the Kneser-S\"uss inequality $(\ref{falseconj})$).
\end{prop}

\proof
If inequality $(\ref{BM})$ (or inequality $(\ref{falseconj})$)
would be true for $F(K)  = \lambda_2(K)$,
by Theorem \ref{BMteo} (or Theorem \ref{KSteo}), the ball would have minimal $\lambda _2$ among planar convex bodies
with a prescribed perimeter.
This contradicts the geometric properties of minimizers for $\lambda _2$ under a perimeter constraint in two dimensions recently shown in
\cite[Theorem 2.5]{BuBuOu}. \qed

\bigskip
\begin{prop}\label{rem3}
In dimension $n=3$, the first Dirichlet eigenvalue of the Laplacian, $\lambda _1 (K)$, does not satisfy
 the Kneser-S\"uss inequality $(\ref{falseconj})$.
\end{prop}
\proof
In this case, if we try repeat the same argument used in the proof of Propositions \ref{rem1} and \ref{rem2}, we do not arrive
any longer to a contradiction, since
the ball actually minimizes $\lambda _1$ among convex bodies with
prescribed surface measure. Nevertheless, if one considers two parallelepipeds of the form
$$P_i := \big \{ (x,y,z) \in \R ^3 \ :\ 0 \leq x \leq x_i ,\ 0 \leq y \leq y_i , \ 0 \leq x \leq z_i \big\}\ ,$$
the values of $\lambda _1 (P_1)$, $\lambda _ 1 (P_2)$, $\lambda _1 ( P _1 \stackrel{.}{+} P _2)$ can be computed explicitly.
And it is a simple exercise to check
that there are values of $x_i, y _i, z_i$ for which (\ref{falseconj}) is false (e.g. $x_1=z_1=z_2=1$, $x_2=4$, $y_2 \rightarrow 0$, $y_1 \rightarrow + \infty$). \qed

\bigskip

\section{Global concavity inequalities in a new algebraic structure}\label{3}

The above counterexamples indicate that the Blaschke addition
is not appropriate to gain concavity for functionals involving Dirichlet energies. This induced us to replace it
by a new algebraic structure on convex bodies. The outcoming concavity and indecomposability results
are proved in Section \ref{3a}, and then some applications in shape optimization are discussed in Section \ref{3b}.

\subsection{The abstract results}\label{3a}

Denoting by $\M ^+ ( \mathbb{S}^{n-1})$  the class of positive measures on $\mathbb{S} ^{n-1}$,
we now introduce the class of maps $\mu: \K _0 ^n \to {\cal M} ^+ ( \mathbb{S} ^{n-1})$
which induce in a natural way an algebraic structure on $\K _0 ^n$.

\begin{definition}\label{crucial1} {\rm We say that a map $\mu:\K _0^n \to \M ^ + (\mathbb{S}^{n-1})$ is a {\it parametrization} of  $\K _0 ^n$ if
\begin{itemize}
\item[(i)] the image $\mu (\K _0 ^n)$ is a convex cone:
$\nu _1, \nu _2 \in  \mu (\K _0 ^n)$, $t_1 , t_2>0$  \ $\Rightarrow$ \
$t_1 \nu_1 + t_2 \nu _2 \in \mu (\K _0 ^n)$;

\item[(ii)] $\mu$ is injective up to translations: $\mu (K _1) = \mu (K_2)$ \ $\Leftrightarrow$ \
$K_1$ and $K _2$ are translates.

\smallskip
\hskip -.9cm We say that $\mu$ is a {\it $(\alpha-1)$-parametrization} if, in addition,

\item[(iii)] $\mu$ is $(\alpha -1)$-homogeneous for some $\alpha \neq 1$:  $K \in \K _0 ^n$, $t >0$\ $\Rightarrow$ \ $\mu (tK) = t ^ {\alpha -1} \mu (K)$.
\end{itemize}

}
\end{definition}
For any parametrization of  $\K _0 ^n$ according to the above definition, the following new operations on convex bodies are well-defined.

\begin{definition}\label{crucial2} {\rm
Let $\mu:\K _0^n \to \M ^ + (\mathbb{S}^{n-1})$ be a parametrization of  $\K _0 ^n$.
For every couple of convex bodies $K, L \in \K _0 ^n$ and any $t >0$,
we define
$K +_\mu L$ and $t \cdot _\mu K$ as the convex bodies determined (up to a translation) by the equalities
 $$\mu( K + _\mu L) =  \mu (K ) + \mu (L) \qquad {\rm and } \qquad \mu (t\cdot _\mu K)= t \mu (K)
\ .$$
}\end{definition}

By abuse of notation, we denote $ K + _\mu 0\cdot _\mu L:=K$.
Several comments on these definitions are in order. We begin by giving some examples of parametrizations.

\begin{example}\label{ex} {\rm
Many $\alpha$-homogeneous functionals $F: \K _0 ^n \to \R ^+$ can be written in the integral form
\begin{equation}\label{effe}
F ( K ) = \frac{1}{|\alpha|} \int h (K) \, d \mu (K)\ ,
\end{equation}
being $\mu (\cdot)$ a $(\alpha -1)$-parametrization of $\K _0 ^n$. Here and in the sequel, when writing the expression at the right hand side\ of (\ref{effe}), we implicitly mean that $\alpha \neq 0$.

Some archetypal cases
which fall in this framework are the following:
\begin{itemize}
\item[(i)] {\it The volume ${\rm Vol} (K)$}: in this case $\alpha = n$
and $\mu (K)$ is the surface area measure of $K$ defined in
(\ref{surfmeasure}).

\item[(ii)] {\it The 2-capacity ${\rm Cap} (K)$}: in this case $n\geq 3$, $\alpha = n-2$ and
 \begin{equation}\label{prima}
\mu (K)[\omega] =  \int_{ \nu _K ^ {-1}(\omega)} |\nabla u _K(x)| ^2 \, d\H ^ {n-1}(x)  \qquad
\hbox{ for all Borel sets } \omega \subseteq \mathbb{S} ^ {n-1}\ ,\end{equation}
where $u _K$ is the electrostatic potential of $K$.

\item[(iii)] {\it The first Dirichlet eigenvalue of Laplacian $\lambda_1 (K)$}: in this case $\alpha = -2$ and
$\mu (K)$ is given by  (\ref{prima}), with $u_K$ equal to the first normalized eigenfunction of $K$.

\item[(iv)] {\it The torsional rigidity $\tau (K)$}: in this case $\alpha = n+2$
and
$\mu (K)$ is again written as in (\ref{prima}), begin now $u_K$ the solution to the equation $- \Delta u = 1$ on $K$ with
$u= 0$ on $\partial K$.

\end{itemize}
For any of these functionals,  the associated map $\mu$ turns out to be a parametrization of $\K_0 ^n$ since it takes values in the class ${\cal A}$
of so-called {\it Alexandrov measures} (positive measures on $\mathbb{S} ^ {n-1}$ which
are not concentrated on any equator and have null barycenter), and satisfies the following Minkowski-type theorem:
for every positive measure $\nu \in {\cal A}$
there exists a convex body $K \in \K _0 ^n$ (unique up to a translation) such that $\nu = \mu (K)$.
In case of the volume, this was established by Minkowski for polyhedra and by Alexandrov in the general case
\cite{Al}. The same result for
capacity and for the first Dirichlet eigenvalue of the Laplacian is due to Jerison (see respectively \cite[Theorem 0.8]{J1}
and \cite[Theorem 7.4]{J2}), whereas for torsional rigidity it has been recently obtained by Colesanti and Fimiani
\cite[Theorems 1 and 2]{CoFi}.  Such results ensure that, for any
of these maps $\mu$, we have $\mu (\K _0 ^n)= {\cal A}$, which is a convex cone, and injectivity holds up to translations.
Thus (i) and (ii) in Definition \ref{crucial1} are satisfied. We remark that also condition (iii) of the same definition is fulfilled
(except for ${\rm Vol} (K)$ in dimension $n=1$, and ${\rm Cap} (K)$ in dimension $n=3$).
}\end{example}

\begin{remark}
{\rm
(i) If $\mu$ is a $(\alpha-1)$-parametrization, by homogeneity we have $t\cdot _\mu K:=  t ^ {1/(\alpha-1)} K$.

(ii) When $\mu$ is the surface area measure, Definition \ref{crucial2} gives exactly the Blasckhe structure.
In spite, as pointed out in \cite{GoSch}, area measures of intermediate order cannot be used
in order to define an addition of convex bodies.

(iii) For any of the parametrizations appearing in Example \ref{ex},
the $\mu$-addition of two convex bodies provides a result geometrically different from their
Minkowski addition. For instance,  in $\R^3$ one can take two regular tetrahedra,
one of which is obtained by the other through a rotation of $\pi/2$, such that
their Minkowski sum has $14$ faces ({\it cf.} \cite{Alex}); clearly their $\mu$-sum, for each of the parametrizations $\mu$ in Example \ref{ex},
has just $8$ faces, whose areas depend on the choice of $\mu$.
}
\end{remark}

\medskip
Relying on the algebraic structure introduced in Definition \ref{crucial2},
a new natural  notion of concavity emerges for the associated integral functionals.

\begin{definition}{\rm Let  $\mu$ be a $(\alpha -1)$-parametrization of $\K_0 ^n$, and let $F : \K _0 ^ n \to \R^+$ be given by $(\ref{effe})$.  We say that $F$ is $\mu$-concave  if
\begin{equation}\label{KS}
F^ {1- (1/\alpha)} (K +_\mu L ) \geq  F ^ {1- (1/\alpha)} (K) + F ^ {1- (1/\alpha)}  (L)\qquad \forall K, L \in \K_0 ^n\ .
\end{equation}
}
\end{definition}

\medskip
We are going to show that this concavity property is strictly related with the Brunn-Minkowski inequality. In particular, it holds true for any of the functionals considered in Example \ref{ex}, since they satisfy
the assumptions of Theorem \ref{abstract1} below (a comprehensive reference for this claim is \cite{Co}).
Thus, we may affirm that (\ref{KS}) it is the natural extension of the Kneser-S\"uss inequality for the volume functional.

\bigskip
\begin{teo}\label{abstract1}
Let $\mu$ be a
$(\alpha -1)$-parametrization $\mu$ of $\K _0^n$, and let $F : \K _0^ n \to \R^+$ be given by $(\ref{effe})$.
Assume that $F$ satisfies the Brunn-Minkowski inequality $(\ref{BM})$ and that
its first variation under Minkowski sums can be written as
\begin{equation}\label{rep1}
\displaystyle{\frac{d}{dt} F(K+ t L) _{| _{t= 0+}} = \frac{\alpha}{|\alpha|}\, \int h (L)  \, d \mu ( K)}\ .
\end{equation}
Then $F$ is $\mu$-concave. Moreover, if equality occurs in $(\ref{BM})$ only when $K$ and $L$ are homothetic, then the same property holds for
$(\ref{KS})$.
\end{teo}

\begin{remark}
{\rm Notice that, if $F: {\cal K} _0^n \to \R ^+$ is a $\alpha$-homogeneous functional which
satisfies (\ref{rep1}), it follows automatically that
it can be written in the integral form (\ref{effe}). This means that only assumption (\ref{rep1}) is crucial. We emphasize that such assumption is less restrictive than it may appear since \eqref{rep1} is nothing else than a representation formula for the shape derivative.
Indeed, if the functional $F$ is shape derivable, under mild assumptions the Hadamard structure theorem for the first shape derivative applies. Consequently, for any deformation by a vector field $V \in C_0^\infty (\R^n; \R^n)$,
the first shape derivative is a linear form depending on the normal component of $V$ on $\partial K$. Moreover, under continuity assumptions, the shape gradient can be identified with a function $g_K \in L^1 (\partial K)$,  so that
 $$\displaystyle{\frac{d}{dt} F\big ((Id +tV)(K) \big ) _{| _{t= 0^+}}}= \int_{\partial K} g_K (V\cdot \nu _K) d{\mathcal H}^{n-1}\ .$$
Though the deformation through the Minkowski addition is in general not induced by a vector field,
  a similar formal analysis leads to representing the derivative in the left hand side of \eqref{rep1} as
 $$\displaystyle{\frac{d}{dt} F(K+ t L) _{| _{t= 0^+}}} =  \int_{\mathbb{S}^{n-1}} \tilde g_K h(L) d{\mathcal H}^{n-1}.$$
Thus, setting $\mu (K):= |\alpha| \tilde  g_K {\mathcal H}^{n-1} \res \mathbb{S} ^ {n-1}$,  we find a map $K\mapsto \mu(K)$ which allows to represent $F$ in the integral form (\ref{effe}). Of course, establishing whether $\mu$ is in fact a parametrization according to Definition \ref{crucial1} is not a trivial question.}
\end{remark}

\proof\hskip -.1cm[of Theorem \ref{abstract1}]
By using (\ref{effe}) and (\ref{rep1}), we may write:
$$\begin{array}{ll} F (K + _\mu L) & = \displaystyle{\frac{1}{|\alpha|} \int h (K + _\mu L) \, d \mu (K + _\mu L)} \\
& = \displaystyle{\frac{1}{|\alpha|} \int h (K + _\mu L) \, d \mu (K ) + \frac{1}{|\alpha|} \int h (K + _\mu L) \, d \mu (L)} \\
& = \displaystyle{\frac{1}{\alpha}\frac{d}{dt} F(K+ t (K +_\mu L)) _{| _{t= 0^+}} }+ \displaystyle{\frac{1}{\alpha} \frac{d}{dt} F(L+ t (K+ _\mu L)) _{| _{t= 0^+}}}
\end{array}
$$
Since by assumption $F$ satisfies the Brunn Minkowski inequality (\ref{BM}), for all $C$ and $C'$ in $\K ^n$ the function
\begin{equation}\label{effino}
f(t):= F ^ {1/\alpha} (C+ tC') - F ^ {1 /\alpha} (C) - t F ^ {1/\alpha} (C')
\end{equation}
is nonnegative for $t \geq 0$. Since $f(0) = 0$, this means that $f' (0) \geq 0$, which gives
$$\frac{1}{\alpha}\frac{d}{dt} F(C+ t C') _{| _{t= 0^+}}
\geq  F^ {1-(1/\alpha)} (C) F ^ {1/\alpha} (C')\ .$$
Applying this inequality once with $C= K$ and $C' = K + _\mu L$, and once with $C= L$ and $C' = K + _\mu L$, we infer
$$F (K + _\mu L) \geq \big [ F^ {1-(1/\alpha)} (K) +  F^ {1-(1/\alpha)} (L) \big ] F ^ {1/\alpha} (K + _\mu L) \ ,$$
and inequality (\ref{KS})  follows dividing by $F ^ {1/\alpha} (K + _\mu L)$.

Assume now that (\ref{KS}) holds with equality sign. By the above proof it follows that, taking $C= K$ and $C' = K + _\mu L$, the function $f$ defined as in (\ref{effino}) satisfies $f' (0) = 0$. But,
by the Brunn Minkowksi inequality (\ref{BM}), we know that $f$ is both nonnegative and concave. Hence it must vanish identically.
By assumption, this implies that
$K$ and $K + _\mu L$ are homothetic, and hence  $K$ and $L$ are homothetic.  \qed
\bigskip

Under the same assumptions of Theorem \ref{abstract1}, following the same lines  one can prove the following refined concavity inequality. Let
$$\sigma(K,L):=\max\Big \{\frac{F(K)}{F(L)}, \frac{F(L)}{F(K)}\Big  \},$$
and $R:\K ^n_0\times \K^n_0 \rightarrow \R^+$ be an ``asymmetry distance'', {\it i.e.}
\begin{itemize}
\item [(i)] $\forall K,L \in   \K^n_0, \quad \forall t,s>0, \quad R(tK,sL)=R(K,L)$,
\item [(ii)]
$\exists c>0, \quad \forall \, K, L, S \in \K ^n_0,  \quad R(K,L)+R(L,S)\ge c R(K,S).$
 \end{itemize}

\begin{prop} Under the same assumptions of Theorem \ref{abstract1}, 
if the Brunn-Minkowski inequality \eqref{BM} holds in the quantitative form
\begin{equation}\label{BMq}
F^{1/\alpha} (K +L) \geq (F ^ {1/\alpha} (K ) + F ^ {1/\alpha} (L)) \left (1+\frac{R(K,L)}{\sigma(K,L)^ {1/\alpha}} \right ) \qquad \forall \, K, L \in \K ^n_0 \ ,
\end{equation}
then  \eqref{KS} holds in the quantitative form
\begin{equation}\label{KSq}
F^ {1- (1/\alpha)} (K +_\mu L ) \geq  (F ^ {1- (1/\alpha)} (K) + F ^ {1- (1/\alpha)}  (L)) \left (1+\frac{c}{2}\frac{R(K,L)}{\sigma(K,L)^ {1-1/\alpha}} \right )  \qquad \forall K, L \in \K_0 ^n\ .
\end{equation}
\end{prop}

In particular,   for $F(K)=\mbox{Vol} (K)$, we obtain a quantitative version of the Kneser-S\"uss inequality. Let us denote for every $K,L \in
 \K_0 ^n$ the Fraenkel relative asymmetry
 $$A(K,L):=\inf_{x_0\in \R^n} \Big \{\frac{\mbox{Vol} (K\Delta (x_0+ \lambda L))}{\mbox{Vol}(K)}  \Big \}, \quad \mbox{where } \lambda :=\Big(\frac{{\rm Vol} (K)}{{\rm Vol}(L)}\Big )^{1/n}.$$
   We notice  that, for every $p \ge 1$, $A(\cdot ,\cdot)^p$ is
an  asymmetry distance which satisfies (i) and (ii) above.
\bigskip
\begin{cor}\label{ksqq}
 There exists a constant $c_n>0$ depending only on the space dimension such that
 $$ {\rm Vol}(K \stackrel{.}{+} L )^ {1- (1/n)} \geq  ( {\rm Vol}(K) ^ {1- (1/n)}+ {\rm Vol}(L) ^ {1- (1/n)})\left (1+\frac{A(K,L)^2}{c_n\sigma(K,L)^ {1-(1/n)}} \right )  \qquad \forall K, L \in \K_0 ^n\ .$$
\end{cor}
\begin{proof}
The proof is a direct consequence of the previous proposition and of the quantitative Brunn-Minkowski inequality proved in \cite{fmp09}.
\end{proof}
\bigskip
As a consequence of Theorem \ref{abstract1}, functionals as in the statement are monotone with respect to the associated parametrization.
\begin{cor}
Let $F$ satisfy the assumptions of Theorem \ref{abstract1}. If $K, L \in \K _0^n$ are such that
$\mu (K) \leq \mu (L)$ (as measures on $\mathbb{S}^{n-1}$), then $F(K) \leq F(L)$.
\end{cor}
\proof We argue as in the proof of \cite[Theorem 7.1]{Alex}. For $t \in (0,1)$, set $\mu_t: = \mu (L) - t \mu (K)$. Since by assumption $\mu (L) \geq \mu (K)$, and since $\mu$ is a parametrization of $\K _0^n$,
there exists a convex body $M_t$ (unique up to a translation), such that $\mu _t = \mu (M_t)$. Then we have
$\mu (L) = \mu (M_t) + t \mu (K)$, which implies $L = M _t + _\mu t \cdot _\mu K$. Then (\ref{KS}) implies
$$F ^{1-(1/\alpha)}(L) = F^{1-(1/\alpha)}\big (M _t + _\mu t \cdot _\mu K \big ) \geq F^{1-(1/\alpha)}(M_t) + t F^{1-(1/\alpha)}(K) \geq t F^{1-(1/\alpha)} (K)\ , $$
and the thesis follows by letting $t$ tend to $1$.
\qed

\bigskip
With a proof similar to the one of Theorem \ref{abstract1},
we obtain the following ``dual'' statement.

\begin{teo}\label{abstract2}
Let $\mu$ be a
$(\alpha -1)$-parametrization of $\K_0 ^n$, and
let $F : \K _0^ n \to \R^+$ be given by $(\ref{effe})$.
Assume that $F$ is $\mu$-concave and that
its first variation under $\mu$-sums can be written as
\begin{equation}\label{rep2}
\displaystyle{\frac{d}{dt}
F(L+_\mu t \cdot _\mu K) _{| _{t= 0^+}} = \frac{\alpha}{|\alpha|} \frac{1}{\alpha -1}\, \int h (L)  \, d \mu ( K)}\ .\end{equation}
Then $F$ satisfies the Brunn Minkowski inequality $(\ref{BM})$. Moreover, if equality occurs in $(\ref{KS})$ only when $K$ and $L$ are homothetic, the same property holds for $(\ref{BM})$.
\end{teo}

\proof By using (\ref{effe}) and (\ref{rep2}), we have:
$$\begin{array}{ll} F (K + L) & = \displaystyle{\frac{1}{|\alpha|} \int h (K +  L) \, d \mu (K +  L)} \\
& = \displaystyle{\frac{1}{|\alpha|} \int h (K ) \, d \mu (K +L) + \frac{1}{|\alpha|} \int h (L) \, d \mu (K+L)} \\
& = \displaystyle{\frac{\alpha-1}{\alpha}\frac{d}{dt}
F(K+_\mu t \cdot_\mu(K + L)) _{| _{t= 0^+}} }+ \displaystyle{\frac{\alpha-1}{\alpha} \frac{d}{dt} F(L+_\mu t\cdot _\mu (K+  L)) _{| _{t= 0^+}}}
\end{array}
$$
Since  now $F$ satisfies (\ref{KS}), for all $C$ and $C'$ in $\K ^n$ the function
\begin{equation}\label{effinobis}
f(t):= F ^ {1-(1/\alpha)} (C+_\mu t\cdot _\mu C') - F ^ {1-(1 /\alpha)} (C) - t F ^ {1-(1/\alpha)} (C')
\end{equation}
is nonnegative for $t \geq 0$. Since $f(0) = 0$, this implies $f' (0) \geq 0$, and hence
$$\frac{\alpha -1}{\alpha}\frac{d}{dt} F(C+_\mu t\cdot _\mu C') _{| _{t= 0^+}}
\geq   F ^ {1/\alpha} (C)F^ {1-(1/\alpha)} (C')\ .$$
Applying this inequality once with $C= K$ and $C' = K +  L$, and once with $C= L$ and $C' = K +  L$, we infer
$$F (K +  L) \geq \big [ F^ {1/\alpha} (K) +  F^ {(1/\alpha)} (L) \big ] F ^ {1-(1/\alpha)} (K +  L) \ ,$$
and  (\ref{BM}) follows dividing by $F ^ {1-(1/\alpha)} (K +  L)$.

If (\ref{BM}) holds with equality sign, the function $f$ defined as in (\ref{effinobis}), with $C= K$ and $C' = K +  L$, satisfies $f' (0) = 0$. Since
by (\ref{KS}) $f$ is both nonnegative and concave, it vanishes identically. By assumption, this implies that
$K$ and $K + L$ are homothetic, and hence $K$ and $L$ are homothetic.   \qed

\bigskip
In view of applications of Theorem \ref{abstract1} to extremal problems, it is useful to give the following
\begin{definition}\rm
Let $\mu$ be a parametrization of $\K _0^n$.
By saying that $K \in \K_0^n$ is $\mu$-{\it indecomposable}, we mean that the equality
$K =  K' + _\mu K'' $ implies that $K '$ and $K ''$ are $\mu$-scalar multiples of $K$.
\end{definition}

When $\mu$ is the surface area measure, a theorem of Bronshtein \cite[Theorem 1]{Br}
characterizes simplexes as the unique Blaschke-indecomposable bodies.
The same proof allows to obtain the following more general statement. We denote by ${\cal A}$ the class of Alexandrov measures on $\mathbb{S} ^ {n-1}$ (defined as in Example \ref{ex}).

\begin{teo}\label{B}  Let $\mu$ be a parametrization of $\K _0^n$, with $\mu (\K_0 ^n) = {\cal A}$. Then a convex body
$K \in \K _0 ^n$ is $\mu$-indecomposable if and only if  ${\rm spt}(\mu (K))$ consists exactly of $(n+1)$ distinct points in $\mathbb{S} ^ {n-1}$.
\end{teo}

\begin{remark} {\rm For any of the parametrizations in Example \ref{ex}, since
${\rm spt}(\mu (K))$ coincides with the image of the Gauss map
$\nu _K$, Theorem \ref{B} tells that the unique $\mu$-indecomposable bodies are simplexes, exactly as it happens for the Blaschke sum.}
\end{remark}

\medskip
\proof\hskip -.1cm[of Theorem \ref{B}]
(i) Let us show that, if ${\rm spt}(\mu (K))$ contains  $(n+2)$ distinct points $\{\xi _1, \dots, \xi _{n+2}\}$
in $\mathbb{S} ^ {n-1}$, then $K$ is not $\mu$-indecomposable.

For every $i  \in \{1, \dots, n+2\}$, let $\omega_i$ be pairwise disjoint neighbourhoods of $\xi _i$ in $\mathbb{S} ^ {n-1}$, and let $C$ be the complement
of their union in $\mathbb{S} ^ {n-1}$. Since $\mu (K)$ has null baricenter, we have
\begin{equation}\label{nb}
\sum_{i = 1} ^ {n+2} \int_{\omega_i} x \, d \mu (K)  + \int_{C} x \, d \mu (K)  = 0\ .
\end{equation}

For generic positive coefficients $\gamma_i$ (to be chosen later), consider the measure
\begin{equation}\label{nu}
\nu := \sum _{i=1} ^ {n+2} \gamma _i \mu \res \omega _i + \frac{1}{2} \mu \res C\ .
\end{equation}

It is clear that, for any choice of $\gamma _i>0$, $\nu$ is not concentrated on any equator of $\mathbb{S} ^ {n-1}$.
If we impose that $\nu$ has null barycenter, we obtain the following system of $n$ equations in the $n+2$ unknown
$(\gamma_1, \dots, \gamma _{n+2} )$:

$$\sum_{i = 1} ^ {n+2} \gamma _i \int_{\omega_i} x \, d \mu (K)  + \frac{1}{2}\int_{C} x \, d \mu (K)  = 0\ .$$

Since by (\ref{nb}) we know that a solution exists (by taking $\gamma_i = 1/2$ for all $i$),
we infer that there exists a $2$-dimensional linear subspace $V$ of $\R^ {n+2}$ such that any $\gamma \in (1/2, \dots, 1/2) + V$ is a solution.
Therefore, we may choose solutions $(\gamma'_1, \dots, \gamma' _{n+2})$ and $(\gamma''_1, \dots, \gamma'' _{n+2})$ such that
$\gamma'_i$ are not all equal, $\gamma'_i>0$, $\gamma'' _i >0$, and $\gamma'_i + \gamma '' _i = 1$.

We set $\nu'$ and $\nu''$ the measures defined as in (\ref{nu}), with $\gamma _i$ equal to $\gamma' _i$ and $\gamma ''_i$ respectively.
Since $\nu'$ and $\nu ''$ belong to ${\cal A}$, and by assumption $\mu$ is a parametrization of $\K _0^n$ with $\mu (\K _0^n) = {\cal A}$,
there exist $K'$ and $K''$ in $\K _0 ^n$ such that $\mu (K' ) = \nu'$ and $\mu (K'') = \nu ''$. By construction, we have
 $\mu (K) = \mu (K') + \mu (K'')$, and $\mu(K')$ is not a multiple of $\mu (K)$. Then $K$ is not $\mu$-indecomposable.

On the other hand, if
${\rm spt}(\mu (K))$ contains strictly less than $n+1$ points in $\mathbb{S} ^ {n-1}$, then $\mu(K)$ cannot belong to ${\cal A}$. Indeed, the null barycenter condition implies that these points are linearly dependent, hence $\mu$ is concentrated on the intersection
of some hyperplane with $\mathbb{S} ^{n-1}$.

Therefore we have proved that, if $K$ is $\mu$-indecomposable, necessarily ${\rm spt}(\mu (K))$ is made exactly by $(n+1)$ distinct points in $\mathbb{S} ^ {n-1}$.

\bigskip

(ii) Viceversa, assume that ${\rm spt}(\mu (K)) =\{\xi _1, \dots, \xi _{n+1}\}$, and let us show that $K$ is $\mu$-indecomposable.
Assume that $K = K ' + _\mu K ''$.
The equality $\mu (K) = \mu (K') + \mu (K'')$ implies that ${\rm spt}(\mu (K'))$ and ${\rm spt}(\mu (K''))$ are contained into
$\{\xi _1, \dots, \xi _{n+1}\}$. For $i = 1, \dots, n+1$, let $\sigma_i$ and $\lambda _i$ be positive numbers such that
$$\mu (K) [\xi _i] = \sigma _i \qquad \hbox{ and } \qquad \mu (K') [\xi _i] = \lambda _i \, \sigma _i \ .$$
Since both $\mu (K)$ and $\mu (K')$ have null barycenter, we have
$$\lambda _1 \sum _{i=1} ^ {n+1} \sigma _i \xi _i = 0 \qquad \hbox{ and } \qquad \sum _{i=1} ^ {n+1} \lambda _i \sigma _i \xi _i = 0\ .$$
By subtraction, we obtain
$$\sum _{i = 2} ^ {n+1} (\lambda _i - \lambda _1) \sigma _i \xi _i = 0\ .$$
If $\lambda _i \neq \lambda _1$ for some $i \in \{2, \dots, n+1\}$, then $\{\xi _2, \dots \xi _{n+2} \}$ would be linearly dependent,
so that $\mu (K)$ would be concentrated on the intersection of $\mathbb{S} ^ {n-1}$ with an hyperplane through the origin,
against the assumption that $\mu(K) \in {\cal A}$.
Therefore it must be $\lambda _i = \lambda _1$ for all
$i \in \{2, \dots, n+1\}$, which means that $K'$ equals $\lambda _1 \cdot _\mu K$.  \qed

\subsection{Isoperimetric-like problems}\label{3b}

We now focus attention on the isoperimetric-like problems of maximizing or minimizing the
dilation invariant quotient
$$\E(K):= \frac{F^ {1- (1/\alpha)} ( K )} {\int d \mu ( K )}  \ ,
$$
being $F$ a $\alpha$-homogeneous functional which can be written under the integral form
$(\ref{effe})$.

\begin{teo}\label{vp} Let $F : \K_0 ^ n \to \R^+$ be given by $(\ref{effe})$ for some
$(\alpha -1)$-homogeneous map $\mu:  \K _0^n \to \M ^+ (\mathbb{S} ^ {n-1})$. Then:
\begin{itemize}
\item[--]
if $F$ satisfies the Brunn-Minkowski inequality
 $(\ref{BM})$, then
the maximum of $\E$ over $\K_0 ^n$ is attained on $\B ^n$; 
\item[--]  if $\mu$ is a parametrization of $\K _0^n$ and $F$ satisfies
 $(\ref{KS})$, with strict inequality for non-homothetic sets, then $\E$ can attain its minimum over $\K _0^n$ only at a $\mu$-indecomposable set.
 \end{itemize}
\end{teo}

\proof
Assume that $F$ satisfies (\ref{BM}).
For a fixed $K \in \K_0 ^n$, let us denote by $B$ the ball in $\K _0^n$ with $\int d \mu (B) = \int d \mu (K)$.
We use once more the argument that the function
$$f(t):= F ^ {1/\alpha} (K+ tB) - F ^ {1 /\alpha} (K) - t F ^ {1/\alpha} (B)
$$
is nonnegative for $t \geq 0$ to deduce that  $f' (0) \geq 0$, which gives
$$ \frac{1}{|\alpha|}\,  h (B) \int  d \mu (K)= \frac{1}{\alpha}\frac{d}{dt} F(K+ t B) _{| _{t= 0^+}}
\geq  F^ {1-(1/\alpha)} (K) F ^ {1/\alpha} (B)\ .$$
Since $|\alpha| = (F(B)) ^ {-1} \, h (B) \int d \mu (B)$, the above inequality yields
$F (B ) \geq   F(K )$.

Assume now that $\mu$ is a parametrization of $\K_0 ^n$ and  that $F$ satisfies (\ref{KS}), with strict inequality for non-homothetic sets. If $K$ can be decomposed as $K = K' + _\mu K''$, with $K'$ and $K''$ non-homotetic, then by arguing as in the proof of Theorem \ref{BMteo} one obtains that $K$ cannot be a minimizer for $\E$.
\qed

\bigskip
As a consequence of Theorem \ref{abstract1}, Theorem \ref{B}, and Theorem \ref{vp}, we deduce

\begin{cor} For any of the functionals in Example \ref{ex},
the unique maximizers of $\E$ over $\K _0^n$ are balls, and
any minimizer of $\E$ over $\K_0 ^n$ (if it exists) is a simplex.
\end{cor}

Let us examine more in detail each case.

\bigskip
\noindent $\bullet$ \ {\bf Volume.}
 When $F(K) = {\rm Vol} (K)$, the inequality $\E (K)
\leq \E (B)$ corresponds to the classical  isoperimetric inequality
for convex bodies, whereas clearly $\inf _{\K_0 ^n} \E = 0$.

\medskip

\noindent $\bullet$ \ {\bf Capacity.} When $F(K) = {\rm Cap} (K)$ ($n\geq 3$), the inequality $\E (K)
\leq \E (B)$ corresponds to
\begin{equation}\label{known}
{\rm Cap} (K) ^{(n-3)/(n-2)} \leq c_n \int _{\partial K} |\nabla u _K |^ 2 \, d {\cal H } ^{n-1} \qquad \forall K \in \K_0 ^n\ ,
\end{equation}
where $u _K$ is the electrostatic potential of $K$ and the dimensional constant $c_n$ is chosen so that
equality holds when $K$ is a ball. Inequality (\ref{known}) was already known: it has been proved by Jerison in
\cite[Corollary 3.19]{J1}.
On the other hand, Lemma 4.13 in the same paper implies that $\E (K _h)$  is infinitesimal for any
sequence $\{K _h\}$ of convex bodies which converges to a convex set $K_0$ contained into a $(n-1)$-hyperplane.
Therefore, we have again $\inf _{\K _0^n} \E = 0$.

\medskip
\noindent $\bullet$ \ {\bf Torsional rigidity.} When $F(K) = \tau (K)$, the inequality $\E (K)
\leq \E (B)$ corresponds to
\begin{equation}\label{unknown}
\tau (K) ^{(n+1)/(n+2)} \leq c_n \int _{\partial K} |\nabla u _K |^ 2 \, d {\cal H } ^{n-1} \qquad \forall K \in \K_0 ^n\ ,
\end{equation}
where now $u _K$ is the solution to the equation $- \Delta u = 1$ on $K$ with
$u= 0$ on $\partial K$,  and  $c_n$ is chosen so that
equality holds when $K$ is a ball.

To the best of our knowledge, the isoperimetric-like inequality (\ref{unknown}) is new. In view of the first variation formula (\ref{rep1}), it can be rephrased as follows: {\it among convex domains with prescribed torsional rigidity,
the ball is the most stable when perturbed by Minkowski addition of a ball.}

\bigskip
On the other hand, we claim that the infimum is again zero.
Indeed (for $n=2$), let $R_l$ denote the rectangle $[0, l] \times [0,1]$. Then the unique solution $u _l$ in $H ^ 1_0 (R_l)$ to the equation
$- \Delta u = 1$ in $R_l$ can be explicitly determined as (see for instance \cite{CrGa})
$$
u _l (x,y) = \frac{lx - x^2}{2} - \frac{4 l ^2}{ \pi ^3} \sum _{k = 0} ^ \infty \frac{\sin [(2k+1) \pi x]}{(2k +1) ^ 3 ( e ^ {(2k+1) \pi/l}+1)}
\{e ^ {(2k+1) \pi y/l} +  e ^ {(2k+1) \pi (1-y)/l} \}\ .$$
From this formula one can easily check that, as $l \to 0$,
$$\tau (R_l) = O ( l ^3) \qquad \hbox{ and } \qquad \int _{\partial R_l} |\nabla u _l| ^ 2 \, d {\cal H} ^1= O ( l^2)\ ,$$
so that
$$\lim _{l \to 0} \frac{\tau (R_l) ^ {3/4}} {\ds \int _{\partial R_l} |\nabla u _l| ^ 2 \, d {\cal H} ^1} = 0\ .$$

\medskip
\noindent $\bullet$\ {\bf The first Dirichlet eigenvalue.} When $F(K) = \lambda _1 (K)$, the inequality $\E (K)
\leq \E (B)$ reads
$$\lambda _1 (K) ^{3/2} \leq c_n \int _{\partial K} |\nabla u _K |^ 2 \, d {\cal H } ^{n-1} \qquad \forall K \in \K_0 ^n\ ,$$
where $u _K$ is the first Dirichlet eigenfunction of $K$ and the dimensional constant $c_n$ is chosen so that
equality holds when $K$ is a ball.

Also this result
seems to be  new and means: {\it among convex domains with prescribed first Dirichlet Laplacian eigenvalue,
the ball is the most stable when perturbed by Minkowski addition of a ball.}

\medskip
In this case the analysis of thinning rectangles leads to guess that the infimum
\begin{equation}\label{open}
\inf _{K \in \K _0 ^2} \frac{\lambda _1 (K) ^ {3/2 }} {\ds \int _{\partial K} |\nabla u _K| ^ 2 \, d {\cal H} ^1}
\end{equation}
remains strictly positive. Indeed
let $R_l$ denote as above the rectangle $[0, l] \times [0,1]$ in $\R^2$, and let now $u _l$ be the first Dirichlet eigenfunction of the Laplacian on  $R_l$.
Then
$$u _l (x,y) = \frac{2}{\sqrt l} \sin \big ( \frac{ \pi x}{l} \big ) \sin ( \pi y  )
\qquad \hbox{ and } \qquad \lambda _1 (R_l) = \pi ^2 \Big ( \frac{1}{l^2} +1\Big ) \ .$$
By direct computation one gets
$$\int _{\partial R_l} |\nabla u _l| ^ 2 \, d {\cal H} ^1= 4 \pi
\Big ( \frac{1}{l^3}+1 \Big )
\ ,$$
so that
$$\lim _{l \to 0} \frac{\lambda _1 (R_l) ^ {3/2 }} {\ds \int _{\partial R_l} |\nabla u _l| ^ 2 \, d {\cal H} ^1} = \frac{ \pi ^2}{4}
\sim 2.46\ .$$

\bigskip
The comparison with another special case shows that the above sequence of rectangles is far from being a minimizing sequence.
Indeed, let $T$ denote the equilateral triangle with vertices $(0,0)$, $(1,0)$, and $(1/2, \sqrt 3 /2)$.
Then (see for instance \cite{Fr})
$$u _T (x,y) = \sin \Big ( \frac{4 \pi y} {\sqrt 3} \Big ) - \sin  \Big (2 \pi \big(x+  \frac{y} {\sqrt 3}\big) \Big ) +
\sin  \Big (2 \pi \big(x-  \frac{y} {\sqrt 3} \big) \Big )\qquad \hbox{ and } \qquad \lambda _1 (T) = \frac{16 \pi ^2}{3}\ .$$
By direct computation one gets
$$\int _{\partial T} |\nabla u _T| ^ 2 \, d {\cal H} ^1= \frac{9}{2} \frac{16 \pi ^2} {3}\ ,$$
so that
$$\frac{\lambda _1 (T) ^ {3/2 }} {\ds \int _{\partial T} |\nabla u _T| ^ 2 \, d {\cal H} ^1} = \frac{4 \pi}{\sqrt 3} \cdot \frac{2}{9} \sim 1.61\ .$$

\section{Local analysis}\label{4}

In this section, we develop local concavity arguments in order to obtain qualitative properties of the optimal sets minimizing (or maximizing)  functionals over $\K _0^n$.
Since the Minkowski addition is not useful for performing {\it local} perturbations of shapes, we use the general framework of deformations by vector fields associated to the classical shape derivative tools, see \cite{HePi} and references therein.
The idea was introduced in \cite{lano10,lanopi10}, where the authors study the convexity constraint for planar shapes: they point out a  {\it local} concavity behavior of the shape functionals, expressed via the second order shape derivative, which implies that optimal convex shapes are polygons. Roughly speaking, this tool allows to deal with the indecomposability concept in a local sense, which naturally leads to identify polygons with extremal sets.

\medskip
Below, we extend  this strategy in any dimension of the space and prove that under a suitable local concavity assumption for the shape functional, optimal convex sets need to be extremal in the sense that their Gauss curvature cannot be positive.
We finally show that this strategy can be applied to a large class of isoperimetric problems, including the conjecture of P\'olya and Sz\"ego.

\subsection{Local concavity of the shape functional}

For the convenience of the reader, we briefly remind the notion of first and second order shape derivatives, focusing on local $\C^2$ deformations, since they are  the ones we use for our purpose. Let $\mathcal{O}$ be a collection of sets in $\R^n$ and $J:\mathcal{O}\to\R$ be a shape functional. If $U$ is a compact set in $\R^n$ and $\varepsilon\in (0,1)$, we denote $\Theta_{U,\varepsilon}=\{\theta\in \C^2(\R^n,\R^n)\;/\;\mbox{Supp}(\theta)\subset U \; and \; \|\theta\|_{2,\infty}<\varepsilon\}$ endowed with the $W^{2,\infty}$-norm.

\begin{definition}\label{def:der}\rm
Let   $K \in\mathcal{O}$, and $U$ a compact set in $\R^n$.
It is said  that $J$ is shape differentiable at $K$ (in $U$) if there exists $\varepsilon>0$ such that $K_{\theta}:=(Id+\theta)(K)\in \mathcal{O}$ for all $\theta\in \Theta_{U,\varepsilon}$ and 
$\J_{K} : \theta\in \Theta_{U,\varepsilon} \mapsto J(K_{\theta})\in\R$
is differentiable at 0.\\
Similarly, $J$ is twice differentiable at $K$ (in $U$) if there exists $\varepsilon>0$ such that $\J_K$ is differentiable in $\Theta_{U,\varepsilon}$, and if $\J_{K}': \Theta_{U,\varepsilon}\to\Theta_{U,\varepsilon}'$ is differentiable at 0.
\end{definition}

The following result refers to the structure of second order shape derivatives (see  \cite{HePi}), and underlines the fact that the normal component of the deformation field plays a fundamental role  in the computation of the shape derivatives.

\begin{prop}
Let $K\in\mathcal{O}$ and $\omega$ a relatively open region of class $\C^3$ contained in $\partial K$. Let $U\subset\R^n$ be a compact set such that $U\cap\partial K\subset\omega$.
If $J$ is twice differentiable at $K$ (in $U$), then there exists a continuous bilinear form $l_{2}^{J}(K):\C^{1}(\omega)\times \C^{1}(\omega)\to\R$ such that,
$$\forall \theta \in \Theta_{U}\textrm{ normal in }\omega,
\quad\J''_K(0)\cdot(\theta,\theta)=l_{2}^{J}(K)\cdot(\varphi,\varphi),$$
where $\varphi:=\theta\cdot\nu_{K}$ is the restriction to $\omega$ of the normal component of $\theta$ (and $\Theta_{U}=\{\theta\in \C^2(\R^n,\R^n)\;/\;{\rm Supp}(\theta)\subset U\}$.
\end{prop}

We are now able to state the main result of this section. 
Roughly speaking it states that, if the second order derivative of the shape functional satisfies a coercivity-like property,   
 its minimizers among convex bodies cannot contain  smooth regions with positive Gauss curvature in their boundary.

Below we denote by $|\cdot|_{H^1(\omega)}$ the classical semi-norm of $H^1(\omega)$, and by $\|\cdot\|_{H^\frac{1}{2}(\omega)}$ the classical norm of $H^\frac{1}{2}(\omega)$.

\begin{teo}\label{th:local}
Let $K^*$ be a minimizer for a functional  $J:\mathcal{K}_{0}^n\to\R$. 
Assume that $\partial K^*$ contains
a relatively open set $\omega$ of class  $\C^3$ such that, 
for every compact set $U$ with $(U\cap\partial K)\subset\omega$, 
 $J$ is twice differentiable at $K^*$ in $U$  and the bilinear form $l_2 ^J(K^*)$ satisfies 
\begin{equation}\label{cont1}
\forall \varphi\in\mathcal{C}^\infty_{c}(\omega),\quad l_{2}^{J}(K^*)\cdot(\varphi,\varphi)\leq -C_{1}|\varphi|^2_{H^1(\omega)}+C_{2}\|\varphi\|^2_{H^{\frac{1}{2}}(\omega)}
\end{equation}
for some constants $C_1>0$, $C_2 \in \R$.
Then $G_{K^*}=0$ on $\omega.$
\end{teo}

\begin{remark}\rm
The presence of regions with vanishing Gauss curvature on the boundary of an optimal convex set was already observed in \cite{BrFeKa} for the Newton problem of minimal resistance (see also \cite{LaPe}). A fundamental difference is that, in our framework,  such a vanishing property of Gauss curvature is obtained by using the coercivity-like property \eqref{cont1} as a key argument, since 
the kind of shape functionals we deal with depend on PDE's ({\it cf.} Theorem \ref{th:PS} below).
\end{remark}

\begin{proof}
Assume by contradiction that $G_{K^*}$ is positive at some point of $\omega$. By continuity, without loosing generality one may assume that $G_{K^*}>0$ on $\omega$.

Let $\varphi\in \mathcal{C}^\infty_{c}(\omega)$, and let $V\in \Theta_{U}$ be an extension  of $\varphi\nu_{K}\in\C^2_{c}(\omega,\R^n)$ to $\R^n$, with compact support $U$ and such that $(U\cap\partial K)\subset\omega$.
All the principal curvatures of $K^*$ being positive in $\omega$, by a continuity argument the principal curvatures of $K_{t}:=(Id+tV)(K^*)$ are still positive for $t$ small enough. Therefore $K_{t}$ is still a convex set, and one can write
$$J(K^*)\leq J(K_{t})\quad\textrm{ for }t\textrm{ small}.$$
Therefore, the second optimality conditions gives $$0\leq l^J_{2}(K^*)\cdot(\varphi,\varphi)$$
and so, using the coercivity property of $l^J_{2}$, we get
$$\forall\varphi\in\mathcal{C}^\infty_{c}(\omega)\ , \quad |\varphi|^2_{H^1(\omega)}\leq \frac{C_{2}}{C_{1}}\|\varphi\|^2_{H^{\frac{1}{2}}(\omega)}. $$
Relying on the density of $\mathcal{C}^\infty_{c}(\omega)$ in $H^1_{0}(\omega)$, we get the continuous imbedding of $H^\frac{1}{2}_{0}(\omega)$ in $H^1_{0}(\omega)$, which is a contradiction.
\end{proof}

\subsection{Extremal problems under surface constraint}

\begin{teo}\label{th:PS}
Let $K^*$ be a minimizer over ${K}^n_{0}$ for a shape functional of the kind
\begin{equation}\label{eq:min_cap}
\mathcal{E}(K):=\frac{F(K)}{S(K)}\ , 
\end{equation}
where $F(K)=f({\rm Vol}(K),\lambda_{1}(K),\tau(K),\Cap(K))$ for some $\C^2$ function $f:\R^4 \to (0,+\infty)$. 

If $\partial K^*$  contains a relatively open set $\omega$ of class $\C^3$, then $G_{K^*}= 0$ on $\omega$.
\end{teo}
\begin{remark}\label{ex:PS}\rm
When $n=3$ and $F(K)=\Cap(K)^2$, the P\'olya-Sz\"ego conjecture claims that the minimizer for  \eqref{eq:min_cap} is the planar disk. As a consequence of Theorem \ref{th:PS}, we obtain that every set $K\in \mathcal{K}^3_{0}$ whose boundary contains a smooth region with  positive Gauss curvature cannot be a minimizer.
\end{remark}
\begin{proof}[of Theorem \ref{th:PS}]
If $K^*$ minimizes $\mathcal{E}$ in $\mathcal{K}_{0}^n$, then  it also minimizes the functional  $K\mapsto J(K):=F(K)- \beta S(K)$ for $\beta:=\frac{F(K^*)}{S(K^*)}>0$. The conclusion of Theorem \ref{th:PS} is then a direct consequence of Theorem \ref{th:local} and of the following lemma: indeed, we get
\begin{eqnarray*}
\forall \varphi\in\mathcal{C}^\infty_{c}(\omega),\quad l_{2}^{J}(K^*)\cdot(\varphi,\varphi)&\leq& -\beta c_{1}|\varphi|^2_{H^1(\omega)}+c_{2}\|\varphi\|^2_{H^{\frac{1}{2}}(\omega)}-\beta c_{2}'\|\varphi\|^2_{L^2(\omega)}\\
&\leq& -\beta c_{1}|\varphi|^2_{H^1(\omega)}+(c_{2}+|\beta c'_{2}|)\|\varphi\|^2_{H^{\frac{1}{2}}(\omega)}
\end{eqnarray*}
with $\beta c_{1}>0$, and $c_{2}$ is obtained differentiating twice $f({\rm Vol}(K),\lambda_{1}(K),\tau(K),\Cap(K))$.
\end{proof}

\begin{lemma}\label{lem:derivative}
Let $F(K)$ be any of the shape functionals ${\rm Vol}(K), \tau(K), \lambda_{1}(K), \Cap(K)$. 
Let $K$  be a bounded open set, and let $\omega\subset \partial K$ be a relatively open region of class $\C^3$.
For any compact set $U$ such that $(U\cap \partial K)\subset\omega$, $F$ and $S$ are twice differentiable in $U$, and
$$\forall \varphi\in \C^2_{c}(\omega), \quad\left|l^{F}_{2}(K)\cdot(\varphi,\varphi)\right|\leq c_{2}\|\varphi\|^2_{H^{\frac{1}{2}}(\omega)}\quad\textrm{ and }\quad l^{S}_{2}(K)\cdot(\varphi,\varphi)\geq c_{1}|\varphi|^2_{H^1(\omega)}+c_{2}'\|\varphi\|^2_{L^2(\omega)}$$
for some constants $c_1>0$, $c_2,c_{2}'\in\R$.
\end{lemma}
\begin{proof}
Except for the Newtonian capacity, the computations of the second order shape derivatives can be found in \cite{HePi}. In our framework,  the only difference is that we assume only the smoothness of  $\omega$, where the local perturbations are performed. This does not affect the results of \cite{HePi}.

For the surface area, one has
$$l^{S}_{2}(K)\cdot(\varphi,\varphi)=\int_{\partial K}|\nabla_{\tau}\varphi|^2-2G_{K}\varphi^2,$$
where $G_{K}$ is the Gauss curvature, well defined on the support of $\varphi\subset\omega$. This easily gives:
$$l^{S}_{2}(K)\cdot(\varphi,\varphi)\geq |\varphi|^2_{H^1(\omega)}+c_{2}'\|\varphi\|^2_{L^2(\omega)}$$
where $c_{2}'=-2\max_{\omega} G_{K}$.\\
For ${\rm Vol},\lambda_{1}$, and $\tau$, formulas given in \cite{HePi} directly give the $H^{\frac{1}{2}}$-continuity property.

Let us provide a complete proof of this result  for the Newtonian capacity, by following the same lines as in \cite{FrGaPi}.

For $\theta\in \Theta_{U,\varepsilon}$ with $\varepsilon$ small enough, we denote $K_{\theta}=(Id+\theta)(K)$, $u_{\theta}=u_{K_{\theta}}$, and $v_{\theta}=u_{\theta}\circ (Id +\theta)$.
We also denote by $\mathcal{D}$ (respectively $\mathcal{D}_{0}$) the closure of the space of smooth functions compactly supported in $\R^3\setminus K$
 (respectively in $\R^3 \setminus\overline{K}$) with respect to the Dirichlet norm $u\mapsto\int_{\R^3\setminus K}|\nabla u|^2$.\\
 If we introduce $\psi\in \mathcal{C}^\infty_{c}(\R^3)$ with
$$0\leq \psi\leq1,\quad \psi\equiv 1 \textrm{ on an open neighborhood of }\overline{K},$$
since $\Delta(\psi-u_{\theta})=\Delta\psi$ on $K_{\theta}$, $w_{\theta}:=\psi\circ(Id+\theta)-v_{\theta}$ is a variational solution of

\begin{equation}\label{eq:diff}
w_{\theta}\in\mathcal{D}_{0}, \quad -\nabla\cdot \left(A(\theta)\nabla w_{\theta}\right)=[f\circ(Id+\theta)]J_{\theta},
\end{equation}
where
$$f=\Delta\psi,\quad A(\theta)=J_{\theta}(Id+D\theta)^{-1}(Id+^t\!\!D\theta)^{-1}, \quad J_{\theta}=det(Id +D\theta).$$

Now, we consider the mapping
$$F:(\theta,w)\in \Theta_{U}\times \left(\mathcal{D}_{0}\cap H^2((\R^3\setminus K)\cap U)\right)\mapsto -\nabla\cdot (A(\theta)\nabla w)-[f\circ(Id+\theta)]J_{\theta}\in\mathcal{D}_{0}'\cap L^2((\R^3\setminus K)\cap U).$$

We check as in \cite{HePi} that $F$ is $\C^\infty$ and that
$$D_{w}F(0,w_{0})=-\Delta:\mathcal{D}_{0}\cap H^2((\R^3\setminus K)\cap U)\to\mathcal{D}_{0}'\cap L^2((\R^3\setminus K)\cap U)$$
is an isomorphism (by Lax-Milgram's Theorem and classical regularity results). By the implicit function Theorem, and unicity in \eqref{eq:diff}, it follows that $\theta\in\Theta_{U}\mapsto w_{\theta}\in\mathcal{D}_{0}\cap H^2((\R^3\setminus K)\cap U)$ is $\C^\infty$ around $\theta=0$, and studying the differentiability of $\theta\mapsto (Id+\theta)^{-1}$ and the composition operator (see \cite[Prop 5.3.10, Exercise 5.2]{HePi} and \cite{La}), one deduce that
$\theta\in\Theta_{U}\mapsto u_{\theta}=\psi-w_{\theta}\circ(Id+\theta)^{-1}\in L^2_{loc}(\R^3)$ is also twice differentiable around 0.\\

Let us now compute the second order derivative of the Newtonian capacity: let $\varphi\in\C^\infty(\omega)$, and let $V\in \Theta_{U}$ be an extension to $\R^n$ of $\varphi\nu_{K}$, with compact support $U$ such that $U\cap\partial K\subset \omega$.
For $t$ small,
 we denote $K_{t}=(Id+tV)(K)$, $u_{t}=u_{K_{t}}$, and $v_{t}=u_{t}\circ (Id +tV)$;
we get that for small enough $t$ such that $\partial K_{t}\subset {\rm Supp}(\psi)$,
$$\Cap(K_{t})=\int_{\R^3\setminus K_{t}}|\nabla u_{t}|^2=-\int_{\partial K_{t}}\frac{\partial u_{t}}{\partial \nu_{t} }=-\int_{\partial K_{t}}\psi\frac{\partial u_{t}}{\partial \nu_{t} }=-\int_{\R^3\setminus K_{t}}u_{t}\Delta\psi=-\int_{\R^3}u_{t}\Delta\psi$$
and so $\frac{d^2}{dt^2}\Cap(K_{t})_{|t=0}=-\int_{\R^3}u''\Delta\psi=\int_{\R^3\setminus{K}}u''\Delta(u-\psi)=\int_{\partial K}u''\frac{\partial u}{\partial \nu}$ with
$$\begin{array}{ll}
\Delta u'=0\; in \;\R^3\setminus \overline{K},&u'+\nabla u\cdot V=0\; on\; \partial K,\\
\Delta u''=0\; in \;\R^3\setminus \overline{K},&u''+2\nabla u'\cdot V+(D^2u\cdot V)\cdot V\;on \;\partial K,
\end{array}
$$
where the boundary conditions are obtained when we differentiate $u_{t}(x+tV(x))=0, \forall x\in\partial K$ (since $V_{|\partial K}$ is supported in $\omega$, these expressions are well defined, since $u$ is regular enough up to $\omega$).\\

Let us denote by
$R:f\in H^{\frac{1}{2}}(\partial K)\longmapsto z=R(f)\in \mathcal{D}$ such that
$\left\{\begin{array}{rl}\Delta z=0&\textrm{ in }\R^3\setminus\overline{K}\\ z=f&\textrm{ on }\partial K\end{array}\right.$ which is a continuous operator.\\
Then
\begin{eqnarray*}
\left|\frac{d^2}{dt^2}\Cap(K_{t})_{|t=0}\right|&=&\left|-\int_{\partial K}\left[2\varphi\frac{\partial u}{\partial \nu}\frac{\partial R\left(-\varphi\frac{\partial u}{\partial \nu}\right)}{d\nu}+\varphi^2\frac{\partial u}{\partial \nu}(D^2u\cdot \nu)\cdot \nu\right]\right|\\
&=&\left|-\int_{K}2\left|\nabla R\left(-\varphi\frac{\partial u}{\partial \nu}\right)\right|^2+\int_{\partial K}\varphi^2(D^2u\cdot \nu)\cdot \nu\right|\\
&\leq&C\left(\left|R(\varphi\frac{\partial u}{d\nu})\right|_{H^1(K)}^2+\|\varphi\|_{L^2(\omega)}^2\right)\\
&\leq&c_{2}\|\varphi\|_{H^{\frac{1}{2}}(\omega)}^2.\\
\end{eqnarray*}

\end{proof}

\par\vskip5mm

\par\vskip5mm

Dorin BUCUR \par
Laboratoire de Math\'ematiques UMR 5127 \par
Universit\'e de Savoie,  Campus Scientifique \par
73376 Le-Bourget-Du-Lac (France)
\par \medskip

Ilaria FRAGAL\`A \par
Dipartimento di Matematica, Politecnico \par
Piazza Leonardo da Vinci, 32 \par
20133 Milano (Italy)
\par \medskip

Jimmy LAMBOLEY \par
Ceremade UMR 7534 \par
Universit\'e de Paris-Dauphine \par
Place du Mar\'echal De Lattre De Tassigny \par
75775 Paris Cedex 16 (France)

\end{document}